\documentclass[a4paper]{amsart}
\usepackage[english]{babel}
\usepackage{inputenc}
\usepackage[T1]{fontenc}
\usepackage{amsmath}
\usepackage{amsthm}
\usepackage{amsfonts}
\usepackage{amssymb}
\usepackage{enumerate}
\usepackage{enumitem}
\usepackage{mathrsfs}

\usepackage{graphicx}
\usepackage{hyperref}
\usepackage{cleveref}
\usepackage{tikz}
\usepackage{braids}
\setlist[enumerate]{label*=\alph*),ref=\alph*)}

\newcommand{\BB}{\mathbb{B}}
\newcommand{\ZZ}{\mathbb{Z}}

\newcommand{\FF}{\mathbb{F}}
\newcommand{\KK}{\mathbb{K}}
\newcommand{\TT}{\mathbb{T}}
\newcommand{\q}{\mathbf{q}}
\newcommand{\p}{\mathbf{p}}

\newcommand{\cF}{\mathcal{F}}
\newcommand{\cR}{\mathcal{R}}
\newcommand{\cS}{\mathcal{S}}
\newcommand{\A}{A}
\newcommand{\Sal}{\overline{\textrm{Sal}}}
\newcommand{\im}{\textrm{im\! }}
\newcommand{\mult}{\textrm{mult}}
\newcommand{\Tor}{\textrm{Tor}}
\newcommand{\coker}{\textrm{coker\! }}
\newcommand{\gr}{\textrm{gr}}
\newcommand{\s}{d}

\newcommand\rightmap[1]{\smash{\mathop{\ \rightarrow\ }\limits^{#1}}}

\newtheorem{teo}{Theorem}[section]
\newtheorem{dfn}[teo]{Definition}
\newtheorem{exam}[teo]{Example}
\newtheorem{prop}[teo]{Proposition}
\newtheorem{lem}[teo]{Lemma}

\newtheorem{obs}[teo]{Remark}

\title{Homology of even Artin kernels}

\author[R.~Blasco]{Rub\'en Blasco-Garc{\'\i}a}

\author[J.I.~Cogolludo]{Jos\'e Ignacio Cogolludo-Agust{\'\i}n}

\author[C.~Mart{\'\i}nez]{Conchita Mart{\'\i}nez-P\'erez}
\address{Departamento de Matem\'aticas, IUMA, Facultad de Ciencias\\
Universidad de Zaragoza\\
c/ Pedro Cerbuna 12\\
E-50009 Zaragoza SPAIN}
\email{rubenb@unizar.es,jicogo@unizar.es,conmar@unizar.es}

\begin{document}

\thanks{
The authors are partially supported by Grupo ``\'Algebra y Geometr{\'\i}a'' of Gobierno de 
Arag\'on/Fondo Social Europeo, the first and second authors are partially supported by 
the Spanish Government MTM2016-76868-C2-2-P, and the first and third authors 
are partially supported by the Spanish Government PGC2018-101179-B-100.} 


\maketitle

\begin{abstract} 
We explicitly compute the homology groups with coefficients in a field of characteristic zero
of cocyclic subgroups or even Artin groups of FC-type. We also give some partial 
results in the case when the coefficients are taken in a field of prime characteristic.
\end{abstract}

\section{Introduction}
The family of Artin-Tits groups has received increasing attention in the past years due to their intrinsic
geometrical nature. They are closely related to Coxeter groups, that is, groups generated by reflections.
Like Coxeter groups, Artin-Tits groups are defined in a combinatorial manner starting from a labeled graph
which describes a presentation. Several questions arise from this fact, trying to determine to what extend 
properties of the groups can be described or characterized combinatorially, that is, in terms of a defining graph. 

Several questions regarding this connection between group and defining graph remain open in general, but are
solved for particular subfamilies of Artin-Tits groups, such as the family of right-angled Artin groups. 
Right-angled Artin groups are defined only by commutation relations among some of their generators. Properties 
such as polyfreeness or residually finiteness are satisfied for these groups, at least those associated with 
finite graphs (see~\cite{Howie,Hermiller-Sunic-PF,Duchamp-PF}). Other important properties are also described 
combinatorially, such as, rigidity, the $K(\pi,1)$-conjecture, quasi-projectivity or the main focus of this 
paper: the homology of Artin kernels.

In short, an Artin kernel is a cocyclic subgroup of an Artin-Tits group. The homology with trivial coefficients 
of such a subgroup can be seen as a module over the ring of deck transformations. Precise definitions will be 
provided in \S\ref{sec:prelim}. 

The main purpose of this paper is to give an explicit combinatorial description of Artin kernels for a family 
of Artin-Tits groups that generalizes right-angled Artin groups, namely, even Artin groups. Our results generalize those in \cite{learymuge}, \cite{ACM1}.

The systematic study of even Artin groups was initiated by the first author in his Ph.D.~thesis~\cite{Blasco-tesis}. 
Some of the results in this thesis were published separately, for example in~\cite{Blasco} there is a 
characterization of the even Artin groups which are quasi-projective in terms of the graph and in~\cite{Blasco-PF} 
it is shown that even Artin groups of FC-type are poly-free. The $K(\pi,1)$-conjecture is also known to be true for 
even Artin groups of FC-type. As precursors of this thesis, one can find some results in the literature about even 
Coxeter groups (see for instance~\cite{Antolin}) and a brief reference to even Artin groups in~\cite{ACM1}.

These groups receive different names in the literature, note that we will refer to them as Artin-Tits groups since 
they are attributed to both E.~Artin and J.~Tits in their full generality. However we use right-angled Artin groups 
and even Artin groups for the subfamilies and Artin kernels for their cocyclic subgroups.

The paper is organized as follows. In \S\ref{sec:prelim} the main definitions of even Artin groups and Artin kernels 
are given together with the construction of the Salvetti complex. The Salvetti complex provides a $K(\pi,1)$-model 
for the group whenever the defining graph is of FC-type. 
In \S\ref{sec:complex} we use this complex to construct, for a given Artin kernel, a chain complex whose homology 
is the homology of that Artin kernel. Using this complex, we prove some partial results about the homology groups 
of Artin kernels with coefficients in an arbitrary field $\KK$. These homology groups are in a natural way modules 
over a polynomial ring $\KK[t^{\pm1}]$ and thus decompose as a direct sum of a free part and a torsion part as follows
\begin{equation}
\label{eq:main-hk}
H_{k+1}(A^\chi_\Gamma;\KK)=\KK[t^{\pm 1}]^{r_k}
\oplus \left( \frac{\KK[t^{\pm 1}]}{(t-1)}\right)^{\dim_\KK\im \partial_{k+1}}
\bigoplus_{\s\in \TT_\Gamma} \bigoplus_{j=1}^\infty 
\left( \frac{\KK[t^{\pm 1}]}{\Phi_\s(t)^{j}}\right)^{n_{k,j}(\s)},
\end{equation}
where $r_k:=\dim_\KK \tilde H_k(\cF^f(\Gamma);\KK)$, $\cF^f(\Gamma)$ is the \emph{finite type flag complex}
associated with $\Gamma$, $n_{k,j}(\s)\in \ZZ_{\geq 0}$, $\Phi_\s(t)$ is the $\s$-th cyclotomic polynomial in 
$\KK[t]$, and $\TT_\Gamma$ is a finite set. This is the main result of Theorem~\ref{teo:torsionhk}.
The last two sections are devoted to calculating both the free and the torsion part in terms of $\Gamma$, $\chi$, and
the characteristic of the base field $\KK$. In particular, we dedicate~\S\ref{sec:resonant} to determining the rank of 
the free part of such modules in the most general case, that is, any Artin kernel and any characteristic for the 
field $\KK$ in Theorems~\ref{teo:h1resonant} and~\ref{teo:h2resonant}. This gives an insight into the problem of 
finiteness properties of Artin kernels in terms of $\Gamma$ and $\chi$ as discussed in Example~\ref{exam:dihedral-rank}.
Finally, in~\S\ref{sec:main} the main results are proved, namely, a combinatorial description of the $k$-th homology 
with coefficients in a field of characteristic zero, of Artin kernels of even Artin groups of FC-type in terms of 
their defining graph. A \emph{non-resonance} condition on the morphism $\chi$ is required for the techniques to work,
namely, $\chi(g_v)\neq 0$ for all the standard generators $g_v$ of the Artin-Tits group $A_\Gamma$.
A first discussion on the torsion part of $H_{1}(A^\chi_\Gamma;\KK)$ in terms of spanning trees is 
provided in Theorem~\ref{thm:ms}. This approach is classical for right-angled Artin kernels and can be applied 
to even Artin kernel as well. The second part of the section deals with the combinatorial description of the 
torsion part of $H_{k+1}(A^\chi_\Gamma;\KK)$ by introducing a multiplicity spectral sequence $\{E^s_\s\}$ of the 
finite type flag complex associated with the graph and with $\s\in \TT_\Gamma$. The main result is provided in
Theorem~\ref{thm:main}, where the sequence of $k$-th relative Euler characteristics of the multiplicity spectral 
sequence at the different pages of $\{E^s_\s\}$, completely determines the invariant factors of the torsion part of 
$H_{k+1}(A^\chi_\Gamma;\KK)$ by a set of linear equations of the following type
$$
\sum_{j\geq s} n_{k,j}(\s)=\chi^{\text{rel}}_{k}(E^s_\s),
$$
where $n_{k,j}(\s)$ is defined in~\eqref{eq:main-hk} for any $d\in \TT_\Gamma$.
Moreover, the Jordan blocks associated with the torsion of $H_{k+1}(A^\chi_\Gamma;\KK)$ have size at most~$k+2$.
We end this paper with an example that shows that the bound provided for the Jordan blocks is sharp.

\subsection{Acknowledgements}
The authors would like to thank the anonymous referee for their comments and suggestions that have helped 
improve the overall quality of this paper.

\section{Preliminaries}
\label{sec:prelim}
For the sake of completeness and to fix notation, we will give the explicit definition of Artin groups. 
To do that we will use the graph notation but note that our graphs are not the Coxeter-Dynkin diagrams of 
Artin groups. Instead, the way in which we define an Artin group associated to a graph generalizes the usual 
graph notation for right-angled Artin groups.

\subsection{Even Artin groups}
Let $\Gamma=(V,E,\ell)$ be a labeled finite simplicial graph. By a label we mean a map $\ell:E\to \ZZ_{> 1}$.
The \emph{Artin-Tits group} associated with $\Gamma$ has the following finite presentation
\begin{equation}
\label{eq:artin-pres}
\A_\Gamma:=\langle g_v, v\in V : R(e,\ell(e)), \ e\in E\rangle,
\end{equation}
where $R(e,\ell(e))$, $e=\{v,w\}$ represents the relation $(g_vg_w)^{k}=(g_wg_v)^{k}$ if $\ell(e)=2k$ and 
$(g_vg_w)^{k}g_v=(g_wg_v)^{k}g_w$ if $\ell(e)=2k+1$.

In this paper we consider positive even labels, that is, $\ell(e)=2\tilde \ell(e)$ with $\tilde \ell(e)\geq1$.
An Artin-Tits group is called \emph{even} (EAG for short) if its associated graph $\Gamma$
has only even labels. In the case of EAGs note that the relations $R(e,\ell(e))$ are commutator relations 
$[a,b]_{k}=(ab)^k(ba)^{-k}$ generalizing $[a,b]=[a,b]_1$.

Analogously, the Coxeter group associated with $\Gamma$ is given as the quotient of $\A_\Gamma$ by the normal subgroup
generated by the squares of the generators $g_v$, that is, 
$$
W_\Gamma:=\A_\Gamma/\langle g_v^2, v\in V\rangle=\langle g_v, v\in V : g_v^2=(g_vg_w)^{\ell(e)}=1, e=\{v,w\}\in E\rangle.
$$
An Artin-Tits group $\A_\Gamma$ is called of \emph{spherical type} if $W_\Gamma$ is a finite group. Another interesting class
of Artin-Tits groups is given as follows. Consider $X\subset \Gamma$ a labeled subgraph of $\Gamma$, $\A_X$ the Artin-Tits 
group associated with $X$ and $V_X\subseteq V$ the set of vertices in $X$. The subgroup of $W_\Gamma$ generated by 
$\{g_v, v\in V_X\}$ is in fact the Coxeter group $W_X$ (see~\cite{Bourbaki}) called the \emph{standard parabolic subgroup} 
of $W_\Gamma$ associated with $X$. An Artin-Tits group $\A_\Gamma$ is said to be of \emph{FC-type} if all the standard 
parabolic groups $W_X$ associated with complete subgraphs $X\subset \Gamma$ are finite.

A standard operation to obtain new Artin-Tits groups from old ones comes from taking the 2-join of two graphs, namely, 
the join of the graphs in which all the new edges are labeled by two. Formally, consider $\Gamma_1=(V_1,E_1,\ell_1)$ and 
$\Gamma_2=(V_2,E_2,\ell_2)$ two labeled graphs and define $\Gamma_1*\Gamma_2=(V,E,\ell)$ as $V:=V_1\cup V_2$, 
$E:=E_1\cup E_2\cup V_1\times_S V_2$ (where $V_1\times_SV_2$ denotes the symmetric product of $V_1$ and $V_2$) and 
$$
\ell(e)=\begin{cases}\ell_i(e) & \textrm{ if } e\in E_i\\ 2 & \textrm{ if } e\in V_1\times_S V_2.\end{cases}
$$

Note that $\A_{\Gamma_1*\Gamma_2}=\A_{\Gamma_1} \times \A_{\Gamma_2}$. We refer to a labeled graph or to its 
associated Artin-Tits group as \emph{irreducible} if it is not the 2-join of two labeled graphs.

\begin{exam}
\label{exam:dihedral}
{\rm
Artin-Tits complete graphs with two vertices are always spherical and their Coxeter groups are dihedral groups of order 
$2\ell$. For $\ell\geq 3$ they correspond to the Dynkin diagrams of types $\mathbb{A}_{2}$, $\BB_2$, resp. 
$\mathbb{I}_{2}(\ell)$ for $\ell=3,4$, resp. $\ell\geq 5$. And their defining graphs in our sense are of the form:
(see~\cite{Coxeter-discrete,Coxeter-complete}).}
$$
\begin{tikzpicture}
\tikzstyle{subj} = [circle, minimum width=8pt, fill, inner sep=0pt]
\tikzstyle{obj}  = [circle, minimum width=8pt, draw, inner sep=0pt]
\node[subj] (n1) at (1,1) {};
\node[subj] (n2) at (2,1) {};
\foreach \from/\to in {n1/n2}
\draw (\from) -- node[above] {$\ell$} (\to) ;
\end{tikzpicture} 
$$
The only irreducible spherical graphs with three vertices are those whose Dynkin diagram is of 
type $\mathbb{A}_{3}$, $\mathbb{B}_{3}$, and $\mathbb{H}_{3}$. Note that none of them yields an even Artin group. 
In particular, the only irreducible EAGs are the cyclic kind $\mathbb{A}_{1}$, the dihedral kind $\mathbb{B}_{2}$, 
and $\mathbb{I}_{2}(2k)$, $k>2$. In other words, any spherical EAG must be a 2-join of these.
\end{exam}

\begin{exam}{\rm
As a consequence of the discussion above, RAAGs are of FC-type and all their subgroups associated to complete 
subgraphs of $\Gamma$ are free abelian.}
\end{exam}

\subsection{Finite type flag complex}
\label{sec:flag}
Artin-Tits groups of FC-type satisfy the $K(\pi,1)$ conjecture (see~\cite{Charney-kpi1}), 
which in particular means that there is a nice combinatorial description of a natural Eilenberg-MacLane space. 
In this section we will briefly describe such spaces for EAGs. This space can be described via a CW-complex 
called the \emph{Salvetti complex} (see~\cite{Charney-finite,Paris}). 

We first define the finite type flag complex $\cF^f(\Gamma)$ of $\Gamma=(V,E,\ell)$ as follows. 
Consider $\cS_{\Gamma}^f:=\{X\subset \Gamma\mid W_X \textrm{ is finite}\}$. To construct $\cF^f(\Gamma)$, a simplex 
of dimension $k$ of $\cF^f(\Gamma)$ is given by any subgraph $X\subset \Gamma$, $V_X=\{v_0,\dots,v_k\}\subset V$ 
such that $X\in \cS_{\Gamma}^f$. We will fix an order in the set of vertices of $\Gamma$. This order yields an 
orientation in each $X\in \cS_{\Gamma}^f$. Let $\KK$ be a field (of arbitrary characteristic at this point). 
The $\KK$-chain complex $C_*^f(\Gamma)$ can be constructed as follows.
\begin{equation}
\label{eq:flag}
\array{l}
C^f_k(\Gamma)={\displaystyle{\sum_{{\tiny{\array{c}X\in S_{\Gamma}^f\\|V_X|=k+1\endarray}}}}} c_X\KK
\endarray
\end{equation}
with differential 
$$\partial_k(c_X)=\sum_{v\in X} \langle X_v|X\rangle c_{X_v},$$
where $X_v$ results from $X$ after deleting the vertex $v\in X$ and $\langle X_v|X\rangle$ is the incidence of $X_v$ in 
$X$ and it is $1$ or $-1$ according to the orientation given in $X$, namely, if $X=\langle v_0,\dots,v_k\rangle$, then
$\langle X_{v_i}|X\rangle=(-1)^i$.

This simplicial complex will be called the \emph{finite type flag complex} 
$\cF^f(\Gamma)$ associated with~$\Gamma$. 

\subsection{The Salvetti complex of an FC-type graph}
\label{sec:salvetti}
The Salvetti complex $\Sal(\Gamma)$ of an FC-type graph $\Gamma$ can be briefly defined as the 2-presentation 
complex associated to the presentation~\eqref{eq:artin-pres} of the Artin-Tits group $\A_\Gamma$ after attaching 
higher dimensional cells for each complete subgraph of~$\Gamma$. As a first approximation its 0-skeleton is given 
by a unique cell $\sigma_\emptyset$, its 1-skeleton is given by closed 1-cells $\sigma_{v}$, $v\in V$, and its 
2-skeleton is given by 2 cells $\sigma_{e}$, $e\in E$. The differential of the Salvetti complex is zero in the case when $\A_\Gamma$ is even. 

\subsection{Artin kernels and the equivariant $\partial^\chi$-complex}
\label{sec:twisted-complex}
Consider a non-trivial morphism $\chi:A_\Gamma\to \ZZ$ for an even Artin group $\A_\Gamma$. 
The kernel of this homomorphism is called the \emph{Artin kernel} of $\A_\Gamma$ associated with $\chi$ and will 
be denoted by $A^\chi_\Gamma$. We will say $\chi$ is \emph{resonant} if $\chi(g_v)= 0$ for some $v\in V_\Gamma$, 
otherwise $\chi$ will be called \emph{non-resonant}.
Let us denote $m_v:=\chi(g_v)$ so that an Artin kernel is represented by a tuple $(m_v)_{v\in V}$.

Note that the abelianization of $\A_\Gamma$ is a free abelian group $H_\Gamma:=\A_\Gamma/\A'_\Gamma$ of 
rank $|V|$ and hence the universal abelian cover of $\Sal(\Gamma)$ is given by a cell decomposition 
$\tilde C_k(\Sal(\Gamma))=C_k(\Sal(\Gamma))\times \KK[H_\Gamma]$ where $\KK[H_\Gamma]=\KK[t_v^{\pm 1},v\in V]$ 
is the group algebra of $H_\Gamma$ over $\KK$. The action of $t_v\in H_\Gamma$ on $g_w$ is given by conjugation 
and it is represented as $t_v*g_w=g_vg_wg_v^{-1}$ and one can check that it does not depend on the choice of 
representative in $H_\Gamma$.

If $\chi$ is surjective, then it determines an infinite cyclic cover of $\Sal(\Gamma)$ which will be denoted as 
$\Sal^\chi(\Gamma)$ and whose chain complex $C_k(\Sal^\chi(\Gamma))$ has
\begin{equation}
\label{eq:action}
C_k(\Sal^\chi(\Gamma))=\tilde C_k(\Sal(\Gamma))\otimes_{\KK[H_\Gamma]}\KK[t^{\pm 1}].
\end{equation}
Here, $\KK[t^{\pm 1}]$ is a $\KK[H_\Gamma]$-module by the action $t_v*1=t^{m_v}$, where $t$ geometrically 
represents the action on the cyclic cover associated with the choice of a generator of $\im \chi$. 

This complex is called the \emph{equivariant $\partial^\chi$-complex} associated with $\Gamma$ and~$\chi$.

\begin{obs}
\label{rem:epi}
From the previous discussion note that, without loss of generality, one can assume that $\chi$ is an epimorphism,
that is $\gcd\{m_v\mid v\in V\}=1$. Otherwise, $\im \chi = d\ZZ$ for $d=\gcd\{m_v\mid v\in V\}$, the action will 
be given by $t^d$ and $C_k(\Sal^\chi(\Gamma))\cong C_k(\Sal^{\frac{1}{d}\chi}(\Gamma))$ where 
$\frac{1}{d}\chi$ is now an epimorphism.
\end{obs}

It is obvious from this description that the universal abelian cover could have been avoided altogether, however, 
it is sometimes more convenient from a conceptual point of view to present the cyclic covers this way. The universal 
abelian cover notation will be used throughout the paper to simplify some formulas.

Since $\Sal(\Gamma)$ is an Eilenberg-MacLane space, $\Sal^\chi(\Gamma)$ is an Eilenberg-MacLane space as well and thus 
$H_k(\Sal^\chi(\Gamma))=H_k(A^\chi_\Gamma)$ is a $\KK[t^{\pm 1}]$-module.

Let us look at the structure of $\Sal^\chi(\Gamma)$ in more detail. First, the 0-skeleton is given by the orbit of a 0-cell
$\sigma^\chi_{\emptyset}$ by the cyclic group $A_\Gamma/\ker\chi$,
that is, $t^n\sigma^\chi_{\emptyset}$, where $\sigma^\chi_{\emptyset}$ is a 
choice of a preimage of $\sigma_{\emptyset}$ by the cyclic cover. Hence $C_0(\Sal^\chi(\Gamma))=\KK[t^{\pm 1}]\sigma^\chi_{\emptyset}$.

The 1-skeleton of $\Sal^\chi(\Gamma)$ is given by the 1-cells $t^n \sigma^\chi_v$, where $\sigma^\chi_v$ is a
choice of 1-cell in the preimage of $\sigma_v$ by the cyclic cover. Note that $\sigma^\chi_v$ is not a closed cell anymore and 
its boundary map is given by $\partial^\chi_1 \sigma^\chi_v=(t_v-1)\sigma^\chi_{\emptyset}=(t^{m_v}-1)\sigma^\chi_{\emptyset}$.

Analogously, the 2-skeleton of $\Sal^\chi(\Gamma)$ is given by the 2-cells $t^n \sigma^\chi_e$, $e=\langle v,w\rangle\in E$ and 
$$
\array{rcl}
\partial^\chi_2 \sigma^\chi_e &=&\left[ (t_v-1) \sigma^\chi_w - (t_w-1) \sigma^\chi_v \right] q_{\tilde\ell(e)}(t_vt_w),\\
&=&\left[ (t^{m_v}-1) \sigma^\chi_w - (t^{m_w}-1) \sigma^\chi_v \right] q_{\tilde\ell(e)}(t^{m_e})
\endarray
$$
where $q_k(x)=\frac{x^{k}-1}{x-1}$, $m_e=m_v+m_w$ and $\ell(e)=2\tilde\ell(e)$. Note that this map is sensitive to the orientation given
to $e=\langle v,w\rangle$.

Finally, the $k$-skeleton of $\Sal^\chi(\Gamma)$ is given by the $k$-cells $t^n \sigma^\chi_X$, $X\in \cS^f$ and 
\begin{equation}
\label{eq:dk}
\array{rcl}
\partial^\chi_k \sigma^\chi_X &=& \sum_{v\in X} \langle X_v|X\rangle (t_v-1) 
\left[\prod_{w\in X_v} q_{\tilde\ell(v,w)}(t_vt_w)\right] \sigma^\chi_{X_v}.\\
\endarray
\end{equation}

\section{On the homology of the equivariant $\partial^\chi$-complex with coefficients in an arbitrary field}
\label{sec:complex}
Recall that $\chi$ is an epimorphism (see Remark~\ref{rem:epi}) and consider the equivariant $\partial^\chi$-complex
$(C^\chi_{*}(\Gamma),\partial^\chi_*)=(C_{*+1}(\Sal^\chi(\Gamma)),\partial_{*+1}^\chi)$ 
associated with $\Gamma$ and~$\chi$ as described in~\S\ref{sec:twisted-complex}
\begin{equation}
\label{eq:dcomplex}
\array{crccc}
...\to & C_k^\chi(\Gamma)
& \buildrel{\partial^\chi_k}\over\to & 
C_{k-1}^\chi(\Gamma)
& \to ... \\
& \sigma^\chi_X & \mapsto & 
\sum_{v\in X} \langle X_v|X\rangle (t^{m_v}-1) \left[ \prod_{w\in X_v} q_{\tilde\ell(v,w)}(t^{m_e}) \right]\sigma^\chi_{X_v}
&
\endarray
\end{equation}

\begin{prop}
\label{prop:isom}
The following isomorphism holds as $\KK[t^{\pm 1}]$-modules
$$
H_k(C_*^\chi(\Gamma),\partial^\chi)=H_{k+1}(\A^\chi_\Gamma;\KK).
$$
\end{prop}

As a consequence, the homology groups $H_{k+1}(\A^\chi_\Gamma;\KK)$ can be seen as $\KK[t^{\pm1}]$-modules. 
Since $\KK[t^{\pm1}]$ is a principal ideal domain, these modules decompose as a torsion part and a free part. 
In the rest of this section we will see how to determine the free part at least in some cases and will prove 
some useful results about the torsion part.

Let us use the following notation. Recall that we are denoting $m_v=\chi(v)$ and $t_v=t^{m_v}$ for $v\in V$ 
and $t_e=t_vt_w$ for $e=\{v,w\}$ an edge in~$\Gamma$.

Define the following resonance set of simplices $\cR(\Gamma,\chi,\KK)=V_\cR\cup E_\cR$, where
$$
V_\cR:=\{v\in V \mid m_v=0\}\quad \quad \textrm{ and } \quad \quad 
E_\cR:=\{e\in E\mid m_e=0 \textrm{ and } \tilde\ell(e)\cdot 1_{\KK}= 0\}.
$$
Note that $t_v-1\not\equiv 0$ if and only if $v\notin V_\cR$ and  $q_{\tilde\ell(e)}(t_e)\not\equiv 0$ 
if and only if $e\notin E_\cR$. 

\begin{dfn} For $X\in \cS^f$, 
$$\p_X:=\prod_{v\in V_X\setminus V_\cR} (t_v-1) \quad \quad \textrm{ and } \quad \quad 
\q_X:=\prod_{e\in E_X\setminus E_\cR} q_{\tilde\ell(e)}(t_e).$$ 
If $\bar X=\{X_1,\ldots,X_r\}$ is a set of elements in $\cS^f$ we also use the following notation
$$\p_{\bar X}:=\prod_{i=1}^r p_{X_i}\quad \quad \textrm{ and } \quad \quad \q_{\bar X}:=\prod_{i=1}^r q_{X_i}.$$ 
\end{dfn}

Then $\p_X\q_X\not\equiv 0$ for any $X\in \cS^f$. Therefore one can formally rewrite~\eqref{eq:dcomplex} as
\begin{equation}
\label{eq:delta}
\frac{1}{\p_X\q_X}\partial^\chi_k\sigma^\chi_X=
\sum_{\tiny{\array{c}Y\subset X\\ |Y|=k\\ \text{non-resonant}\endarray}} \langle Y|X\rangle \frac{1}{\p_Y\q_Y}\sigma^\chi_Y,
\end{equation}
where the sum is taken over the non-resonant $Y\subset X$, that is, $v\notin V_\cR$ for any $v\in X\setminus Y$, 
and $e=\{v,w\}\notin E_\cR$ for any $v\in X\setminus Y$ and $w\in Y$.

\begin{dfn} 
We say $\chi$ is \emph{$\KK$ non-resonant} if $\cR(\Gamma,\chi,\KK)=\emptyset$. 
\end{dfn}

Note that if $\chi$ is $\KK$ non-resonant, for $X\in \cS^f$, one obtains
$$\p_X:=\prod_{v\in V_X} (t_v-1)\quad \quad \textrm{ and } \quad \quad \q_X:=\prod_{e\in E_X} q_{\tilde\ell(e)}(t_e).$$ 

Note that if $\KK$ has characteristic zero, then a character is $\KK$ non-resonant if and only if it is non-resonant.

For the rest of this section we will fix a field $\KK$ of arbitrary characteristic and 
assume that the character $\chi$ is $\KK$ non-resonant.

\subsection{The free part in the $\KK$ non-resonant case}
\begin{teo}
\label{teo:freepart}{\cite{Papadima-Suciu-Toric}}
Let $\chi:A_\Gamma\to \ZZ$ be a $\KK$ non-resonant epimorphism. If $A^\chi_\Gamma:=\ker \chi$, then 
the free part of $H_{k+1}(A^\chi_\Gamma;\KK)$ as a $\KK[t^{\pm 1}]$-module has rank 
$r_k:=\dim_\KK \tilde H_k(\cF^f(\Gamma);\KK)$.
\end{teo}

\begin{proof}
For simplicity let us denote by $F=\cF^f(\Gamma)$ the finite type flag complex of $\Gamma$, by $F_k=\cF^f_k(\Gamma)$ its 
set of $k$-simplices, and by $C_k=C^f_k(\Gamma)$ the free abelian group generated by $F_k$. Note that 
$\left( \frac{1}{\p_X\q_X}\sigma^\chi_X\right)_{X\in F_k}$ is a basis of the vector space $C_k\otimes \KK(t)$. 
Analogously, $\left( \sigma_X\right)_{X\in F_k}$ a basis of $C_k\otimes \KK$. Note that both spaces have the same
dimension over their respective fields, moreover by~\eqref{eq:delta} and since $\chi$ is $\KK$ non-resonant, 
both boundary maps are given by the incidence matrix $(\langle Y|X\rangle)_{X,Y}$, where $\langle Y|X\rangle$ is defined 
for $X\in C_k$, $Y\in C_{k-1}$ and is given as in~\eqref{eq:dk} if $Y\subset X$ and as 0 otherwise. Hence the result follows.
\end{proof}

\subsection{A resolution matrix}
We will use the notation $M^\chi_k(t)$ to denote the matrix of the homomorphism $\partial^\chi_k$ of 
$\KK[t^{\pm 1}]$-modules with respect to the natural bases $C^\chi_k(\Gamma)$ and $C^\chi_{k-1}(\Gamma)$. 
Also, $M_k$ will represent the matrix with respect to the analogous basis over~$\KK$ of the 
homomorphism $\partial_k$ of the complex 
$(C^\chi_*(\Gamma),\partial^\chi_*)$ 
defined in~\S\ref{sec:flag}.

In order to give formulas for the torsion we will study the set of invariants of the matrices $M^\chi_k(t)$.
This is a consequence of the following well-known result~\cite{ACLMM-Artin}

\begin{lem}
\label{lem:fitting}
The torsion part of $H_{k+1}(A^\chi_\Gamma;\KK)$ coincides with the torsion part of $\coker  \partial^\chi_{k+1}$.
In particular, the non-trivial invariant factors of $M^\chi_{k+1}(t)$ determine the torsion part 
of~$H_{k+1}(A^\chi_\Gamma;\KK)$.
\end{lem}

\begin{proof}
The short exact sequence
$$
0\to H_{k+1}(A^\chi_\Gamma;\KK)=H_k(C^\chi_*(\Gamma),\partial^\chi_*) \to \frac{C^\chi_k(\Gamma)}{\im \partial^\chi_{k+1}}
\to \frac{C^\chi_k(\Gamma)}{\ker \partial^\chi_{k}}\cong \im \partial^\chi_{k}\to 0
$$
follows from Proposition~\ref{prop:isom}. The right-most term is free since it is a submodule of $C^\chi_{k-1}(\Gamma)$,
which is a free module over a PID, hence the first part follows. The second part is a consequence of the structure 
theorem for modules over a PID and the fact that $M^\chi_{k+1}(t)$ is the free presentation matrix 
of~$\coker \partial^\chi_{k+1}$.
\end{proof}

\subsection{The Fitting ideals of $M^\chi_{k+1}(t)$}
By Lemma~\ref{lem:fitting} it is enough to calculate the invariant factors of $M^\chi_{k+1}(t)$, or equivalently, 
its Fitting ideals. Recall that the $s$-th Fitting ideal $I_s$ associated with an $R$-module $U$ is given as the 
ideal generated by the minors of size $r\times r$ for $r=m-s$ of any free presentation matrix $M$ of $U$, that is, 
$R^n\rightmap{M} R^m\to U\to 0$. To see this, recall that 
$$I_s\subseteq I_{s+1}.$$
If $R$ is a PID and we write $I_s=f_sR$, then we have $f_{s+1}\mid f_s$ and the invariant factors of $M$ are the 
elements $g_s:={f_s\over f_{s+1}}$. These ideals yield the usual a decomposition of $U$ as a sum of a free $R$-module 
and modules of the form $R/g_sR$.

Note that a square submatrix of $M^\chi_{k+1}(t)$ of size $r\times r$ is determined  
by the choice of $r$ $(k+1)$-simplices and $r$ 
($k$)-simplices. We will denote such a submatrix by $S_{(\bar X,\bar Y)}$, where $\bar X=\{X_1,\dots,X_r\}$ 
(resp. $\bar Y=\{Y_1,\dots,Y_r\}$) is a list of $(k+1)$-simplices (resp. ($k$)-simplices).
We define by $\mathfrak{m}^\chi_{(\bar X,\bar Y)}$ (resp. $\mathfrak{m}_{(\bar X,\bar Y)}$) the minors 
$\det(S_{(\bar X,\bar Y)})$ of the matrix $M^\chi_{k+1}(t)$ (resp. $M_{k+1}$). One has the following immediate properties.

\begin{prop}
\mbox{}
\begin{enumerate}
\item\label{proper:minor1}
If $Y_i\not\subset X_j$ for some $j$ and any $i=1,\dots,r$, then 
$\mathfrak{m}^\chi_{(\bar X,\bar Y)}=\mathfrak{m}_{(\bar X,\bar Y)}=0$.
Analogously, if $Y_i\not\subset X_j$ for some $i$ and any $j=1,\dots,r$, then 
$\mathfrak{m}^\chi_{(\bar X,\bar Y)}=\mathfrak{m}_{(\bar X,\bar Y)}=0$.
\item\label{proper:minor2} 
If $\bar X$ contains a $(k+1)$-cycle, that is, $\sigma:=\sum_{i=1}^r \lambda_i \sigma_{X_i}$ for some non-trivial choice 
$\lambda_1,\dots,\lambda_n\in \KK$ and $\partial^\chi_{k+1}(\sigma)=0$, then 
$\mathfrak{m}^\chi_{(\bar X,\bar Y)}=\mathfrak{m}_{(\bar X,\bar Y)}=0$.
\end{enumerate}
\end{prop}

In order to characterize the choices of $(\bar X,\bar Y)$ whose associated minor $\mathfrak{m}^\chi_{(\bar X,\bar Y)}$ is
non zero we need the following. 

\begin{dfn} 
Let $\bar X$ (resp. $\bar Y$) be a list of ($k+1$)-simplices (resp. $k$-simplices) in the finite type 
flag complex $\cF^f$. Consider 
$N(\bar X):=\cF^f_{k}\bigcup \cup_{\sigma_{k+1}\in \bar X} \sigma_{k+1}$
and 
$N(\bar Y^c):=\cF^f_{(k-1)}\cup \bigcup_{\sigma_{k}\notin \bar Y} \sigma_{k}$.
We say $(\bar X,\bar Y)$ is \emph{$(k+1)$-acyclic of order $r$} (or simply \emph{acyclic of order $r$}) if 
$(N(\bar X),N(\bar Y^c))$ is acyclic and $|\bar X|=|\bar Y|=r$.
\end{dfn}

Using the notation above one has the following result on the minors of $M^\chi_{k+1}(t)$ in~$\KK[t]$.

\begin{prop}
\label{prop:mXY}
Let $\chi:A_\Gamma\to\ZZ$ be a $\KK$ non-resonant morphism and let $\mathfrak{m}^\chi_{(\bar X,\bar Y)}$ be a minor of 
size $r\times r$ of the matrix $M^\chi_{k+1}(t)$ associated with the pair $(\bar X,\bar Y)$. Then
\begin{enumerate}[label=\textrm{(\roman*)}]
 \item \label{prop:mXY1}
 $$\mathfrak{m}^\chi_{(\bar X,\bar Y)}=\frac{\p_{\bar X}\q_{\bar X}}{\p_{\bar Y}\q_{\bar Y}}\mathfrak{m}_{(\bar X,\bar Y)}.$$
 \item \label{prop:mXY2}
 $\mathfrak{m}^\chi_{(\bar X,\bar Y)}$ is non-zero if and only if $(\bar X,\bar Y)$ is acyclic of order $r$.
 \item \label{prop:mXY3}
The biggest possible size $r$ such that $\mathfrak{m}^\chi_{(\bar X,\bar Y)}\neq 0$ is $r=\dim_\KK\im \partial_{k+1}$.
\end{enumerate}
\end{prop}

\begin{proof}
Part~\ref{prop:mXY1} is an immediate consequence of~\eqref{eq:delta} since $\chi$ is $\KK$ non-resonant. 
To prove part~\ref{prop:mXY2} note that 
$$C_i(N(\bar X),N(\bar Y^c))=\begin{cases}  
C_{k+1}(\bar X) & \textrm{ if } i=k+1,\\ 
C_{k}(\bar X,\bar Y^c) & \textrm{ if } i=k\\
0 & \textrm{ otherwise }\end{cases}$$
since $\bar X$ (resp. $\bar Y^c$) has dimension $k+1$ (resp. $k$). Hence $H_{k+1}(N(\bar X),N(\bar Y^c);\KK)=0$ is equivalent to 
asking $C_{k+1}(\bar X;\KK)\hookrightarrow C_{k}(\bar X,\bar Y^c;\KK)$. 
However, since they both have the same dimension and 
$\det(S_{(\bar X,\bar Y)})=\mathfrak{m}_{(\bar X,\bar Y)}$ this is in fact
an isomorphism whose associated matrix is $S_{(\bar X,\bar Y)}$ and the result follows.
Finally, the rank of $M_{k+1}^\chi$ is given by the dimension of $\im_\KK \partial_{k+1}$ and 
thus part~\ref{prop:mXY3} follows.
\end{proof}

\subsection{Torsion in $H_{k+1}(A^\chi_\Gamma;\KK)$}
Note that the discussion about the Fitting ideals above together with  Proposition~\ref{prop:mXY}\ref{prop:mXY1} 
imply that the torsion part $H_{k+1}(A^\chi_\Gamma;\KK)$ can be described in terms of  the $\s$-th cyclotomic polynomials 
$\Phi_\s(t)$ for $\KK$ for $\s$ dividing  either $m_v$ or $\tilde \ell(e)m_e$. 
Moreover, Proposition~\ref{prop:mXY}\ref{prop:mXY1} is also true even if the character is $\KK$ resonant. 
However, we consider the $\KK$ non-resonant case only to give a more detailed description in our next result.

\begin{teo}
\label{teo:torsionhk}
Consider $H_{k+1}(A^\chi_\Gamma;\KK)$ as a $\KK[t^{\pm 1}]$-module where $\chi:A_\Gamma\to \ZZ$ is $\KK$ non-resonant. 
Then 
$$
H_{k+1}(A^\chi_\Gamma;\KK)=\KK[t^{\pm 1}]^{r_k}
\oplus \left( \frac{\KK[t^{\pm 1}]}{(t-1)}\right)^{\dim_\KK\im \partial_{k+1}}
\bigoplus_{\s\in \TT_\Gamma} \bigoplus_{j=1}^\infty 
\left( \frac{\KK[t^{\pm 1}]}{\Phi_\s(t)^{j}}\right)^{n_{k,j}(\s)},
$$
for some $n_{k,j}(\s)\in \ZZ_{\geq 0}$, where $\Phi_\s(t)$ is the $\s$-th cyclotomic polynomial in $\KK[t]$,
$r_k:=\dim_\KK \tilde H_k(\cF^f(\Gamma);\KK)$, and
\begin{equation}
\label{eq:torus}
\array{lcl}
\mathbb T_\Gamma & = & \mathbb T_{V_\Gamma}\cup \mathbb T_{E_\Gamma},\\
\mathbb T_{V_\Gamma}&=&\bigcup_{v\in V_\Gamma}\{\s\in \ZZ_{>1} \mid m_v=0 \mod \s\}, \textrm{ and }\\
\mathbb T_{E_\Gamma}&=&\bigcup_{e\in E_\Gamma} \{\s\in \ZZ_{>1} \mid \tilde \ell(e)m_e=0 \mod \s, \textrm{ but } m_e\neq 0 \mod \s\}.
\endarray
\end{equation}
\end{teo}

\begin{proof}
The free part was given in Theorem~\ref{teo:freepart}. By Proposition~\ref{prop:mXY}\ref{prop:mXY1} the only
possible torsion appears as a root of polynomials of type either $\p_{\bar X}$ or $\q_{\bar X}$, that is, 
as roots of either $t^{m_v}-1$ or $q_{\tilde \ell(e)}(t^{m_e})$.
The union of the first type of roots is given by $\mathbb T_{V_\Gamma}$
whereas the second type of roots is given by $\mathbb T_{E_\Gamma}$.
To end the proof, note that the hypothesis that $\chi$ is $\KK$ non-resonant together with 
Proposition~\ref{prop:mXY}\ref{prop:mXY1} imply that the polynomial $(t-1)$ is a factor with multiplicity 
precisely $r$ of each $r\times r$ non-zero minor $\mathfrak{m}^\chi_{(\bar X,\bar Y)}$. Together with the 
discussion above this implies that the $(t-1)$-part of the torsion module is semisimple. Moreover,
according to  Proposition~\ref{prop:mXY}\ref{prop:mXY3}, the biggest possible such $r$ is 
$\dim_\KK\im \partial_{k+1}$, hence the result follows.
\end{proof}

\section{On the free part of $H_{k+1}(A^{\chi}_{\Gamma};\KK)$ in the resonant case}
\label{sec:resonant}

In this section we will see how to adapt $\Gamma$ to apply Theorem~\ref{teo:freepart} in the $\KK$ resonant case for 
$H_i(A^{\chi}_{\Gamma};\KK)$, $i=1,2$. The case of higher homology groups will be treated in a forthcoming paper.
Let $\chi$ be a character and $\cR(\Gamma,\chi)$ its resonant set. In the following particular cases of edges $e=\{v,w\}$ one has
$$
\partial^\chi(\sigma^\chi_e)=
\begin{cases}
\langle e|w\rangle\p_w\q_e\sigma^\chi_v & \textrm{ if } v\in V_\cR, w\notin V_\cR\\
0 & \textrm{ if either } v,w\in V_\cR \textrm{ or } e\in E_\cR.\\
\end{cases}
$$
Consider the graph $\Gamma_1$ obtained from $\Gamma$ after deleting the (open) edges $e=\{v,w\}\in E$ 
for which either $v,w\in V_\cR$ or $e\in E_\cR$ and the vertices $v\in V_\cR$ whose link intersects 
$V\setminus V_\cR$. There is a morphism between the (flag) chain complexes associated to $\Gamma$ and 
$\Gamma_1$ which is a quasi-isomorphism up to $H_0$. Moreover, $\chi$ induces a character 
$\chi_1:A_{\Gamma_1}\to\ZZ$ which produces another Artin kernel $A^{\chi_1}_{\Gamma_1}$ such that 
$H_1(A^{\chi}_{\Gamma};\KK)\cong H_1(A^{\chi_1}_{\Gamma_1};\KK)$. The character $\chi_1$ might still be 
$\KK$ resonant for $\Gamma_1$, but the vertices in the kernel of $\chi_1$ are isolated in $\Gamma_1$. 
Hence the proof in Theorem~\ref{teo:freepart} can be applied to~$\Gamma_1$ to obtain the following result.

\begin{teo}
\label{teo:h1resonant}
Let $A^\chi_\Gamma$ be the Artin kernel of a general character $0\neq \chi:A_\Gamma\to \ZZ$ of an even Artin-Tits group.
Then the rank of the free part of $H_1(A^\chi_\Gamma;\KK)$ as a $\KK[t^{\pm 1}]$-module is~$\dim_\KK \tilde H_0(\Gamma_1;\KK)$. 

In particular, $H_1(A^\chi_\Gamma;\KK)$ is a torsion module if and only if $\Gamma_1$ is connected.
\end{teo}

The previous theorem, when applied to right-angled Artin groups, recovers 
well-known results~\cite{Bux-Gonzalez-Bestvina,MMVW}, 
characterizing the group $A^\chi_\Gamma$ to be finitely generated 
(which implies $H_1(A^\chi_\Gamma;\KK)$ is torsion) if and only if $\Gamma\setminus V_\cR$ is connected and 
$\chi$ is dominant, that is, for any $v\in V_\cR$ there exists a $w\in V_\Gamma\setminus V_\cR$ such that 
$e=\{v,w\}\in E_\Gamma$, which is equivalent to asking $\Gamma_1$ to be connected.

Analogously, for the following particular cases of sets $X=\{u,v,w\}\in \cS^f$ note that

$$
\partial^\chi(\sigma^\chi_X)=
\begin{cases}
\langle X_u|X\rangle\p_u\q_e\sigma^\chi_e & \textrm{ if } v,w\in V_\cR, u\notin V_\cR\\
\langle X_u|X\rangle\p_u\q_e\sigma^\chi_e & \textrm{ if } e=\{v,w\}\in E_\cR, u\notin V_\cR\\
0 & \textrm{ otherwise}.\\
\end{cases}
$$
The simplicial subcomplex $\cF_2$ obtained from the 2-skeleton of the finite type flag complex $\cF^f(\Gamma)$ 
after removing the 2-cells $X=\{u,v,w\}\in \cF^f(\Gamma)$ such that either $\{u,v\}\in E_\cR$ and $w\in V_\cR$ or 
$u,v,w\in V_\cR$, removing the 1-cells $e=\{v,w\}$ in $E_\cR$ or $v,w\in V_\cR$ whose link intersects 
$V_\Gamma\setminus V_\cR$ and then identifying the ends of the remaining 1-cells $e=\{u,v\}$ such that 
$u,v\in V_\cR$ or $e\in E_\cR$. As before one obtains a morphism of complexes which induces 
$H_2(A^{\chi}_{\Gamma};\KK)\cong H_1(C^\chi_*(\cF_2);\KK)$. In that case, $\cF_2$ might not be the finite type 
flag complex of an Artin-Tits group. However, the proof of Theorem~\ref{teo:freepart} 
still applies to obtain the following result.

\begin{teo}
\label{teo:h2resonant}
Let $A^\chi_\Gamma$ be the Artin kernel of a general character $0\neq \chi:A_\Gamma\to \ZZ$ of an even Artin-Tits group.
Then the rank of the free part of $H_2(A^\chi_\Gamma;\KK)$ as a $\KK[t^{\pm 1}]$-module is $\dim_\KK \tilde H_1(\cF_2;\KK)$. 

In particular, $H_2(A^\chi_\Gamma;\KK)$ is a torsion module if and only if $\cF_2$ is 1-acyclic.
\end{teo}

\begin{exam}
\label{exam:dihedral-rank}
{\rm To illustrate the different behavior of the homology groups according to whether the character is or not $\KK$ resonant, 
we can consider just the example of a dihedral Artin-Tits group.
Let $\Gamma$ be the complete graph with two vertices as in Example~\ref{exam:dihedral} with label 4. 
Put $V_\Gamma=\{u,v\}$, $E_\Gamma=\{e=\{u,v\}\}$, and let $\chi:A_\Gamma\to \ZZ$ be the character defined by 
$\chi(g_u)=1$, $\chi(g_v)=-1$.
Then $\chi$ is $\KK$ non-resonant if and only if $\textrm{char}(\KK)\neq 2$. 
Note that $A_\Gamma=\langle g_u,g_v \mid (g_ug_v)^2=(g_vg_u)^2\rangle$
and $A^\chi_\Gamma=\langle w_n=g_u^{n+1}g_vg_u^{-n} \mid w_n^2=w_0^2, n\in \ZZ\rangle$ and hence
$$H_1(A^\chi_\Gamma;\KK)=
\begin{cases}
\KK[t^{\pm 1}] & \textrm{if } \textrm{char}(\KK)=2, \\
\frac{\KK[t^{\pm 1}]}{(t-1)} & \text{otherwise.}
\end{cases}
$$
Theorems~\ref{teo:freepart} and \ref{teo:torsionhk} hold for $\textrm{char}(\KK)\neq 2$, since 
$\tilde H_k(\cF^f(\Gamma);\KK)=0$ ($\Gamma$ is contractible) and $\dim_\KK\im \partial_{1}=1$ 
($\partial_{1}$ is injective and $C_1^f(\Gamma)=c_e\KK$).

Note that if $\textrm{char}(\KK)=2$, then Theorems~\ref{teo:freepart} and~\ref{teo:torsionhk} do not hold, but one can apply 
the discussion in \S\ref{sec:resonant}. In this case, the free part of $H_1(A^\chi_\Gamma;\KK)$ comes from the fact that 
$\tilde{\Gamma}_1$ is not connected, since it results from $\Gamma$ after removing the edge.

In addition, the non-finiteness presentation of $A^\chi_\Gamma$ is a consequence of the non-triviality of the free part of 
$H_1(A^\chi_\Gamma;\FF_2)$.

Also observe that the same argument applies to all dihedral Artin-Tits groups with edge labeled by 
$2\ell$ for the character $\chi$ defined as above, i.e., that 
$H_1(A^\chi_\Gamma;\KK)$ is not a torsion module whenever $\textrm{char}(\KK)|\ell$, thus showing that $A^\chi_\Gamma$ 
does not admit a finite presentation.}
\end{exam}

\begin{exam}{\rm
Consider the Artin-Tits group associated with the graph $\Gamma$ shown in Figure~\ref{fig:ejemplocuadradoresonante}
and the character $\chi:A_\Gamma\to \ZZ$ defined as $\chi(v_1)=\chi(v_2)=\chi(v_3)=1$, and $\chi(v_0)=-1$.

\begin{figure}[ht]
\begin{tikzpicture}
\tikzstyle{pto} = [circle, minimum width=8pt, fill, inner sep=0pt]
\node[pto] (n1) at (3,3) {};
\node[pto] (n2) at (0,0) {};
\node[pto] (n3) at (3,0) {};
\node[pto] (n4) at (0,3) {};
\node (n5) at (1.4,1.7) {$4$};
\draw (1.5,0) node[below] {$2$} ;
\draw (1.5,3) node[above] {$2$} ;
\draw (3,1.5) node[right] {$2$} ;
\draw (0,1.5) node[left] {$2$} ;
\draw (n1) node[above right] {$v_2$} ;
\draw (n2) node[below left] {$v_0$} ;
\draw (n3) node[below right] {$v_3$} ;
\draw (n4) node[above left] {$v_1$} ;
\draw (n2) -- (n3) -- (n1) -- (n4) -- (n2) -- (n1);
\end{tikzpicture}
\caption{}
\label{fig:ejemplocuadradoresonante}
\end{figure}

Note that $\chi$ is $\FF_2$ resonant with resonance set  $\cR(\Gamma,\chi,\FF_2)=E_\cR=\{\sigma_{02}\}$.
Its equivariant complex can be described as
$$
\array{ccccccccccc}
0 & \to & C^\chi_3(\Gamma) & \to & C^\chi_2(\Gamma) & \to & C^\chi_1(\Gamma) & \to & C^\chi_0(\Gamma) & \to & 0\\
&& \sigma_{012} & \mapsto & (t+1)\sigma_{02} && \sigma_{0}&\mapsto& (t^{-1}+1)\sigma_{\emptyset}\\
&& \sigma_{023} & \mapsto & (t+1)\sigma_{02} && \sigma_{i}, i=1,2,3& \mapsto& (t+1)\sigma_{\emptyset}\\
&&&& \sigma_{02} & \mapsto & 0 &&\\
&&&& \sigma_{12} & \mapsto & (t+1)(\sigma_{1}+\sigma_{2}) &&\\
&&&& \sigma_{23} & \mapsto & (t+1)(\sigma_{2}+\sigma_{3}) &&\\
&&&& \sigma_{01} & \mapsto & (t+1)(\sigma_{0}+t^{-1}\sigma_{1}) &&\\
&&&& \sigma_{03} & \mapsto & (t+1)(\sigma_{0}+t^{-1}\sigma_{3}) &&\\
\endarray
$$
Hence 
$$
\begin{aligned}
H_1(A^\chi_\Gamma;\FF_2)&=
\left( \frac{\FF_2[t^{\pm 1}]}{(t+1)}(\sigma_{1}+\sigma_{2})\right)\oplus 
\left(\frac{\FF_2[t^{\pm 1}]}{(t+1)}(\sigma_{2}+\sigma_{3})\right)\oplus 
\left(\frac{\FF_2[t^{\pm 1}]}{(t+1)}(\sigma_{2}+t\sigma_{0})\right)\\
H_2(A^\chi_\Gamma;\FF_2)&=\left( \frac{\FF_2[t^{\pm 1}]}{(t+1)}\sigma_{02}\right)\oplus 
\FF_2[t^{\pm 1}]((\sigma_{12}+\sigma_{23})+t(\sigma_{01}+\sigma_{03})).
\end{aligned}
$$
Note that $\dim_{\FF_2} H_1(A^\chi_\Gamma;\FF_2)<\infty$ but $\dim_{\FF_2} H_2(A^\chi_\Gamma;\FF_2)=\infty$ 
is consistent with the discussion above since $\cF_2=\Gamma_1=\Gamma\setminus \{\sigma_{02}\}$ 
is connected but not simply connected.

For the same graph but a different character $\chi':A_\Gamma\to \ZZ$ defined as $\chi'(v_1)=\chi'(v_3)=0$, $\chi'(v_2)=1$, 
and $\chi'(v_0)=-1$ note that $\Gamma_1=\Gamma\setminus\{\sigma_{02}\}$ as before, however $\sigma_{02}$ remains in
the construction of $\cF_2$ as described above. In fact,
$\cF_2=\mathbb{S}^1\vee\mathbb{S}^1\vee\mathbb{S}^1$, which implies,
$$
\begin{aligned}
H_1(A^{\chi'}_\Gamma;\FF_2)&=
\left( \frac{\FF_2[t^{\pm 1}]}{(t+1)}\sigma_{1}\right)\oplus 
\left(\frac{\FF_2[t^{\pm 1}]}{(t+1)}\sigma_{3}\right)\oplus 
\left(\frac{\FF_2[t^{\pm 1}]}{(t+1)}(\sigma_{2}+t\sigma_{0})\right)\\
H_2(A^{\chi'}_\Gamma;\FF_2)&=\FF_2[t^{\pm 1}]\sigma_{02}\oplus 
\FF_2[t^{\pm 1}](\sigma_{12}+t\sigma_{01})\oplus 
\FF_2[t^{\pm 1}](\sigma_{23}+t\sigma_{03}).
\end{aligned}
$$}
\end{exam}

\section{Torsion in $H_{k+1}(A^\chi_\Gamma;\KK)$ for $\textrm{char}(\KK)=0$ and $\chi$ non-resonant}
\label{sec:main}
The purpose of this section is to give more specific formulas for the invariants $n_{k,j}(\s)$ as introduced in
Theorem~\ref{teo:torsionhk} in the particular case $\textrm{char}(\KK)=0$. Note that in this case, the notion 
of non-resonant and $\KK$ non-resonant are equivalent.

We will assume in this section that $\Gamma$ is connected. Otherwise $A_\Gamma=A_{\Gamma_1} * A_{\Gamma_2}$ 
is a free product of groups and one can check that the corresponding equivariant complexes fit in a short exact sequence 
$$0\to C^{\chi_1}_*(\Gamma_1)\to \bar C^{\chi}_*(\Gamma)\to C^{\chi_2}_*(\Gamma_2)\to 0$$
where $\chi_i$ are the corresponding restrictions of $\chi$ to $\Gamma_i$ and $\bar C^{\chi}_*(\Gamma)$ is a variation of 
$C^{\chi}_*(\Gamma)$ where we replace $C^{\chi}_{-1}(\Gamma)$ by 
$\bar C^{\chi}_{-1}(\Gamma)=C^{\chi_1}_{-1}(\Gamma_1)\oplus C^{\chi_2}_{-1}(\Gamma_2)$. This implies that as abelian groups 
$$H_{k+1}(A^\chi_\Gamma;\KK)=H_{k+1}(A^{\chi_1}_{\Gamma_1};\KK)\oplus H_{k+1}(A^{\chi_2}_{\Gamma_2};\KK)$$
for $k>0$. However, as a word of caution, their submodule structure depends on the fact that $\chi_i$ is a restriction of $\chi$,
for instance, in case $\chi_i$ is not an epimorphism. In the case of $k=0$, by Lemma~\ref{lem:fitting} we get the same 
decomposition for the $\KK[t^{\pm1}]$-torsion submodules.

First, a discussion for $H_{1}(A^\chi_\Gamma;\KK)$ is included based on the structure of spanning trees of $\Gamma$ 
following the ideas in~\cite{ACM1}. A second discussion for the general $H_{k+1}(A^\chi_\Gamma;\KK)$ requires 
the introduction of the multiplicity spectral sequence.

\subsection{The $H_{1}(A^\chi_\Gamma;\KK)$ case}
In order to calculate $n_{0,j}(\s)$ we will obtain the invariants of the matrix $M_1^\chi(t)$ which 
defines the boundary map of the $\partial^\chi$-complex with respect to the natural basis, as described in the 
previous section.

Note that $\dim_\KK\im \partial_2=|V_\Gamma|-1$ and hence maximal 1-acyclic pairs $(\bar X,\bar Y)$ are given by 
$\bar X$ the set of edges of a spanning tree $T$ and $\bar Y^c$ the choice of a vertex $v_1\in V_{T}=V_\Gamma$.
(observe that the number of vertices in any tree is exactly one plus the number of edges) 
We will call $(T,v_1)$ a \emph{rooted spanning tree}. 
All 1-acyclic pairs $(\bar X,\bar Y)$ of size $r=|V_\Gamma|-s$ can be obtained as follows. Consider 
$F_s=T_1\cup...\cup T_s$ an $s$-forest, that is, a disjoint union of $s$ trees. By convention, a tree with zero edges is 
just a vertex. An $s$-forest $F_s$ is called a spanning $s$-forest if the union of its vertices is $V_\Gamma$, that is, 
$V_\Gamma=\cup V_{T_i}$. A pair $(F_s,\bar v)$ for $\bar v=(v_1,...,v_s)$, where $v_i\in V_{T_i}$ is called a 
\emph{rooted spanning $s$-forest} of~$\Gamma$. The following result is immediate.

\begin{lem}
\label{lem:rooted}
Any 1-acyclic pair $(\bar X,\bar Y)$ of order $r=|V_\Gamma|-s$ is given as $\bar X=E_{F_s}$ and $\bar Y^c=\bar v$ 
for a rooted spanning $s$-forest $F_s$ of~$\Gamma$. Moreover, if $(F'_{s+1},\bar v'\})$ is obtained from $(F_{s},\bar v)$
after eliminating an edge, $(\bar X,\bar Y)$ (resp. $(\bar X',\bar Y')$) denotes their corresponding 1-acyclic pairs, and 
$m$ (resp. $m'$) denotes the multiplicity of a root $\zeta_\s$ of $\Phi_\s(t)$ in the polynomial 
$\mathfrak m^\chi_{(\bar X,\bar Y)}$ (resp. $\mathfrak m^\chi_{(\bar X',\bar Y')}$), then $m\leq m'+2$.
\end{lem}

\begin{proof}
The first part is immediate. For the moreover part, let us denote by $\bar v'=(v_1,...,v_s,v_{s+1})$, 
where $\bar v=(v_1,...,v_s)$ and by $e\in E_{T_s}$ the edge removed to obtain the $(s+1)$-forest $F'_{s+1}$. 
Note that from Proposition~\ref{prop:mXY}\ref{prop:mXY1} one obtains
$$\mathfrak m^\chi_{(\bar X,\bar Y)}=
\mathfrak m^\chi_{(\bar X',\bar Y')} \frac{\p_e\q_e}{\p_{v_{s+1}}}.$$
Since the polynomials $\p_e$, $\q_e$ only have simple roots, the claim follows.
\end{proof}

We can use this to obtain a more detailed description of $H_1(A^\chi_\Gamma;\KK)$ as follows.

\begin{teo}
\label{thm:ms} The invariants $n_{0,j}(\s)$ of Theorem \ref{teo:torsionhk} vanish for $j> 2$ so 
$$
H_1(A^\chi_\Gamma;\KK)=\KK[t^{\pm 1}]^{r}
\bigoplus_{\s\in \mathbb T_\Gamma} 
\left[
\left( \frac{\KK[t^{\pm 1}]}{\Phi_\s(t)}\right)^{n_{0,1}(\s)}
\oplus
\left( \frac{\KK[t^{\pm 1}]}{\Phi_\s(t)^{2}}\right)^{n_{0,2}(\s)}
\right].
$$
\end{teo}

\begin{proof}
Define for each $1\leq s\leq |V_\Gamma|$, 
$$f_s=\gcd \left( \frac{\p_{F_s}\q_{F_s}\p_{\bar v}}{\p_{V}}: (F_s,\bar v) \textrm{ is a rooted spanning } s\textrm{-forest}\right)$$
where by $\p_{F_s}$, $\q_{F_s}$ we denote $\p_{\bar X}$, $\q_{\bar X}$ where $\bar X$ is the set of edges of the forest $F_s$.

By Proposition~\ref{prop:mXY}\ref{prop:mXY2} and Lemma~\ref{lem:rooted} the ideals $I_s=(f_s(t))$ 
are the Fitting ideals rank $s$. Recall that the invariant factors of 
$H_1(A^\chi_\Gamma;\KK)$ are obtained as $f_s/f_{s+1}$.
The result follows by observing that Lemma~\ref{lem:rooted} implies that the difference between the multiplicity of a root $\zeta_d$ of $\Phi_\s(t)$ in the polynomials $f_s$ ans $f_{s+1}$ is at most 2.
\end{proof}

The rest of the section will be devoted to calculating the invariants $n_{k,j}(\s)$ in terms of the graph 
$\Gamma$ and a non-resonant epimorphism~$\chi$.

\subsection{A weight map}
For a polynomial $f\in\KK[t^{\pm 1}]$, let $\mult_\s(f)$ denote the biggest integer $m$ such that $\Phi_\s(t)^m\mid f$, 
or equivalently the multiplicity of a primitive $d$-th root of unity $\zeta_d$ as a root of $f$.
From Theorem~\ref{thm:ms} we see that the invariants $n_{0,1}(\s)$ and $n_{0,2}(\s)$ can be computed by computing 
$\mult_\s(f_s)$ for each $1\leq s\leq |V_\Gamma|$. For example, in the case of $f_1$ this multiplicity is
\begin{equation}
\label{eq:multiplicity}
\mult_\s(f_s)=\min \left\{\mult_\s\left(\frac{\p_v(t)\p_T(t)\q_T(t)}{\p_V(t)}\right) \mid 
(T,v) \text{ is a rooted spanning tree of } \Gamma\right\},
\end{equation}

Assume $\s\in \mathbb T_\Gamma$ (in particular $d\neq 1$), then 
$$
\array{rcl}
\mult_\s(\p_v)\geq 1 & \iff & d\mid m_v\\
\mult_\s(\q_e)\geq 1 & \iff & \begin{cases}
d\mid \tilde\ell_em_e \\
d\nmid m_e
\end{cases}
\endarray
$$

The following result describes the multiplicity $\mult_\s(\p_e\q_e)$.

\begin{lem}
\label{lem:multiplicity}
Under the conditions above, $\mult_\s(\p_e\q_e)\leq 2$.
\end{lem}

\begin{proof}
Assuming $\mult_\s(\p_e\q_e)\geq 1$, these are the possibilities for~$\s$:
\begin{enumerate}
\item \label{cases:mult1}
If $d\mid m_v$ and $d\mid m_w$, then $\mult_\s(\p_e)=2$ and $\mult_\s(\q_e)=0$, since $d\mid m_e$.
\item \label{cases:mult2}
Otherwise, if say $d\mid m_v$ and $d\mid \tilde \ell_em_e$, then $\mult_\s(\p_e\q_e)=\mult_\s(\p_e)+\mult_\s(\q_e)=1+1=2$. 
\item \label{cases:mult3}
Otherwise, if say $d\mid m_v$ but $d\nmid \tilde \ell_em_e$, then $\mult_\s(\p_e\q_e)=\mult_\s(\p_e)=1$.
\item \label{cases:mult4}
Finally, if $d\nmid m_v$, $d\nmid m_w$, $d\mid \tilde \ell_em_e$, and $d\nmid m_e$, then $\mult_\s(\p_e\q_e)=\mult_\s(\q_e)=1$.
\end{enumerate}
This proves the claim.
\end{proof}

The multiplicity maps on the vertices $V_\Gamma\to \ZZ_{\geq 0}$, $v\mapsto \mult_\s(\p_v)$ and edges 
$E_\Gamma\to \ZZ_{\geq 0}$, $e\mapsto \mult_\s(\p_e\q_e)$ determine the invariant factors of $H_1(A^\chi_\Gamma;\KK)$.
We generalize this to a \emph{weight map} $w:\cF^f(\Gamma)\to \ZZ_{\geq 0}$ on the finite type flag complex $\cF^f(\Gamma)$ 
defined as $w(X):=\mult_\s(\p_X\q_X)$, which will play an important role in the general case.

\subsection{The multiplicity spectral sequence}
Given $\s\in \TT_\Gamma$ consider $w:\cF^f(\Gamma)\to \ZZ_{\geq 0}$, $w(X):=\mult_\s(\p_X\q_X)$ 
as defined above. Associated with this weight map one can construct the standard increasing weight filtration of 
simplicial complexes $F_{d,*}$ as follows:
$$F_{d,p}C_q:=\langle X\in \cF^f_q(\Gamma)\mid w(X)\leq p\rangle\subset C^f_q(\Gamma),$$
that is, generated by the $q$-simplices $X\in\cF^f(\Gamma)$ for which $\Phi_\s$ has multiplicity at most $p$ in 
the polynomial $\p_\sigma\q_\sigma$. For convenience, if $p<0$, then $F_{d,p}C_q=\{0\}$.
Note that $F_{d,p}C_q\subset F_{d,p+1}C_q$ and $\partial F_{d,p}C_{q+1}\subset F_{d,p}C_q$. The spectral sequence 
associated with this filtration starts with the term $E^0_{d,(p,q)}:=F_{d,p}C_{p+q}/F_{d,p-1}C_{p+q}$ and 
$\partial^0:E^0_{d,(p,q)}\to E^0_{d,(p,q-1)}$ is well defined by $\partial_{*}$ since 
$\partial_{p+q} F_{d,p}C_{p+q}\subset F_{d,p}C_{p+q-1}$ and 
$\partial_{p+q} F_{d,p-1}C_{p+q}\subset F_{d,p-1}C_{p+q-1}$. 
Note that $E^1_{d,(p,q)}=H_{p+q}(F_{d,p}C_{*})$ and the spectral sequence $\{(E^k_{d,(p,q)},\partial^k)\}$ is well defined
where $\partial^k$ is a morphism of type $(-k,k-1)$. This weight map in more generality can be found in~\cite{SV:13}.

\begin{prop}
\label{prop-ss}
The multiplicity spectral sequence $(E^k_{d(p,q)},\partial^k)$ associated with the flag complex $\cF^f(\Gamma)$ of 
FC-type is bounded and satisfies the following properties:
\begin{enumerate}[label=\textrm{(\roman*)}]
 \item \label{prop-ss-1}
 $E^0_{d,(p,q)}=\{0\}$ if $p+q<-1$
 \item \label{prop-ss-2}
 $E^0_{d,(p,-p-1)}=\begin{cases}
             \langle\sigma_\emptyset\rangle_{\KK} & \textrm{ if } p=0\\
             \{0\} & \textrm{ otherwise }
             \end{cases}$
 \item \label{prop-ss-3}
 $E^0_{d,(p,q-p)}=\begin{cases} 
      \langle \sigma_X \mid X\in \cF^f_q(\Gamma), w(X)=p\rangle_{\KK} & \textrm{ if } q\geq 0, p\leq q+1\\
      \{0\} & \textrm{ otherwise. }\\
      \end{cases}$
\end{enumerate}
\end{prop}

\begin{proof}
Part~\ref{prop-ss-1} (resp.~\ref{prop-ss-2}) is an immediate consequence of 
$C^f_k(\Gamma)=\{0\}$ if $k<-1$ ($C^f_{-1}(\Gamma)=\langle\sigma_{\emptyset}\rangle$). 
Part~\ref{prop-ss-3} can be proved by induction. The first step, $q=1$, is  Lemma~\ref{lem:multiplicity}. 
Let $m=w(X')$, and let us consider $X=X'\cup\{v\}$ a $(q+1)$-simplex. If $w(v)=1$, then the only new edges 
$e=\{v,v'\}$ such that $\mult_\s(\q_e)=1$ are those for which $\mult_\s(\p_{v'})=0$. Moreover, by the FC-type condition,
if $\mult_\s(\q_e)=1$ with $e=\{v,v'\}$, then for all $e'=\{v',v''\}$ one has $\mult_\s(\q_{e'})=0$. 
Hence there are at most $q+1-m$ of such edges. 
Thus $w(X)= w(v)+\sum_{v'\in X'} \mult_\s(\q_e)+w(X')\leq 1+(q+1-m)+m=q+2$.
As a consequence of~\ref{prop-ss-1} and~\ref{prop-ss-3}, $E^0_{d,(p,q)}=\{0\}$ if $p<0$, $p+q<-1$, and $p+q>w(\Gamma)$
the clique number of $\Gamma$, that is, the dimension of $\cF^f(\Gamma)$. Hence the multiplicity spectral sequence is bounded.
\end{proof}

\subsection{The $(\Phi_\s)$-adic filtration}
In order to study the primary part of $H_{k+1}(A^\chi_\Gamma;\KK)$, in the decomposition given 
in Theorem~\ref{teo:torsionhk} we will fix $\s\in \TT_\Gamma$, denote by $\Phi_\s(t)$ the cyclotomic polynomial of order $\s$, 
and consider $\hat \Lambda$ the completion of $\Lambda=\KK[t^{\pm 1}]$ with respect to the $(\Phi_\s)$-adic filtration as defined 
in~\cite[\S$4.2$]{Papadima-Suciu-Toric}. If $\KK_d=\Lambda/(\Phi_\s)$ denotes the residue field and 
$\iota:\Lambda\hookrightarrow \hat\Lambda$ the natural inclusion, then $\gr(\hat \Lambda)\cong \KK_d[\tau]$ and if 
$f\in \Lambda$ is a polynomial, then $\iota(f)$ is a unit in $\hat \Lambda$
if and only if its class in $\KK_d$ is non-trivial, that is, $\Phi_\s\not |f$. 

Alternatively, one can work with $\hat\Lambda:=\KK_\s[t]_P$ where the subindex ${}_P$ means localization at the ideal $P$ 
which is the ideal generated by $\tau=t-\zeta_\s$ with $\zeta_\s\in\KK_\s$ root of $\Phi_\s(t)$.
\begin{prop}
\label{prop:ns}
Under the above conditions
$$
\begin{aligned}
\dim_{\KK_d} \left(H_{k}(C^\chi_*(\Gamma);\hat\Lambda)\otimes_{\hat\Lambda}\frac{\hat\Lambda}{(\tau^s)}\right)
&=sr_k+\sum_{j=1}^{s-1} jn_{k,j}(d)+\sum_{j\geq s} s n_{k,j}(d)\\
\dim_{\KK_d} \Tor^{\hat\Lambda}_1\left( H_{k}(C^\chi_*(\Gamma);\hat\Lambda),\frac{\hat\Lambda}{(\tau^s)}\right)
&=\sum_{j=1}^{s-1} jn_{k,j}(d)+\sum_{j\geq s} s n_{k,j}(d).
\end{aligned}
$$
\end{prop}

\begin{proof}
It follows immediately from 
$$
\Tor^{\hat\Lambda}_1\left( \frac{\hat\Lambda}{(\tau^{s_1})},\frac{\hat\Lambda}{(\tau^{s_2})} \right)
=\frac{(\tau^{M})}{(\tau^{s_1+s_2})}
$$
and 
$$
\frac{\hat\Lambda}{(\tau^{s_1})}\otimes_{\hat\Lambda} \frac{\hat\Lambda}{(\tau^{s_2})}=
\frac{\hat\Lambda}{(\tau^{m})},
$$
where $M:=\max\{s_1,s_2\}$, $m:=\min\{s_1,s_2\}$.
\end{proof}

In order to describe the homology of the Artin kernels we need to introduce some notation associated with invariants 
of the multiplicity-spectral sequence. Let us denote $h^s_{d,(p,q)}:=\dim_\KK E^s_{d,(p,q)}$, 
$h^s_{\s,q}:=\sum_{p\geq 0} h^s_{d,(p,q-p)}$, and 
$\chi^{\text{rel}}_{k}(E^s_\s):=\sum_{q=0}^k (-1)^{k-q}(h^s_{\s,q}-h^\infty_{\s,q})$, this is the 
\emph{$k$-th relative Euler characteristic} of $E^s_\s$. Note that these are combinatorial invariants of the flag 
complex $\cF^f(\Gamma)$ and the weight map.

\begin{prop}
\label{prop:hpq}
$$
\dim_{\KK_\s} H_{k}( C^\chi_*(\Gamma); {\hat\Lambda}/{(\tau^{s})})=\sum_{j=1}^s h^j_{\s,k}.
$$
Moreover, 
$(E^\bullet_{\s,(p,q)}\otimes_{\hat\Lambda} \frac{\hat\Lambda}{(\tau^{s})},d^\bullet)$
degenerates at the $s$-th page.
\end{prop}

\begin{proof}
The spectral sequence $(E^\bullet_{\s,(p,q)}\otimes_{\hat\Lambda} \frac{\hat\Lambda}{(\tau^{s})},d^\bullet)$ 
associated with the complex $C^\chi_q(\Gamma)\otimes_{\hat\Lambda} \frac{\hat\Lambda}{(\tau^{s})}$, where 
$$\partial(\sigma^\chi_X\otimes 1)=\sum_{X_v}\langle X_v|X\rangle\tau^{w(X)-w(X_v)}\sigma^\chi_{X_v}\otimes 1$$
and the decreasing filtration $F^q=(\tau^q)\hat\Lambda\cdot C^\chi_p(\Gamma)$ is bounded and such that $d^q=0$ if $q\geq s$.
By an argument generalizing~\cite[Corollary 5.5]{Papadima-Suciu-Toric} the abutment of this spectral sequence is 
$H_{k}( C^\chi_*(\Gamma); {\hat\Lambda}/{(\tau^{s})})$. Moreover, if $\sigma^\chi_X\in F_{\s,p}C^\chi_{q}(\Gamma)$, then 
$$\partial(\sigma^\chi_X\otimes 1)=\sum_{i=0}^{s-1} \partial_i(\sigma^\chi_X\otimes 1)\tau^i,$$
where $\partial_i(\sigma^\chi_X\otimes 1)\in F_{\s,p-i}C^\chi_{q-1}(\Gamma)$. 

For $s=1$, the first page of this spectral sequence coincides with that of $(E^\bullet_{\s,(p,q)},d)$ and it degenerates 
at this page, hence $\dim_{\KK_\s} H_{k}( C^\chi_*(\Gamma); {\hat\Lambda}/{(\tau)})=\sum_{p+q=k} h^1_{\s,(p,q)}=h^1_{\s,k}$.
Similarly, $\dim_{\KK_\s} H_{k}( C^\chi_*(\Gamma); \gr^j(\hat\Lambda))=\sum_{p+q=k} h^j_{\s,(p,q)}=h^j_{\s,k}$.
The result follows by induction and the K\"unneth formula.
\end{proof}

\subsection{The general $H_{k+1}(A^\chi_\Gamma;\KK)$ case}
The previous discussion provides the following formula for the invariants of~$H_{k+1}(A^\chi_\Gamma;\KK)$.

\begin{teo}
\label{thm:main}
Under the previous notation,
\begin{equation}
\label{eq:thm1}
r_k+\sum_{j\geq s} n_{k,j}(d)=\sum_{p\geq 0} h^s_{\s,(p,k-p)}- \sum_{j\geq s} n_{k-1,j}(d).
\end{equation}
Equivalently,
\begin{equation}
\label{eq:thm2}
\sum_{j\geq s} n_{k,j}(d)=\chi^{\text{rel}}_{k}(E^s_\s).
\end{equation}
This determines completely the $\Phi_\s$-primary part of~$H_{k+1}(A^\chi_\Gamma;\KK)$.

Moreover, the Jordan blocks associated with the torsion of $H_{k+1}(A^\chi_\Gamma;\KK)$ have size at most~$k+2$.
\end{teo}

\begin{proof}
By the universal coefficient theorem $H_{k}(C^\chi_*(\Gamma);\hat\Lambda)=H_{k}(C^\chi_*(\Gamma);\KK)\otimes_\KK \hat\Lambda$ and 
\begin{equation}
\label{eq:UCTh}
0\to H_{k}(C^\chi_*(\Gamma);\hat\Lambda)\otimes_{\hat\Lambda} {\hat\Lambda}/{(\tau^{s})}\to
H_{k}(C^\chi_*(\Gamma);{\hat\Lambda}/{(\tau^{s})})\to 
\Tor^{\hat\Lambda}_1\left( H_{k-1}(C^\chi_*(\Gamma);\hat\Lambda),{\hat\Lambda}/{(\tau^{s})} \right)\to 0.
\end{equation}
The result follows from comparing~\eqref{eq:UCTh} for $s$ and $s-1$ and using Propositions~\ref{prop:ns} and \ref{prop:hpq}.
By Proposition~\ref{prop:ns}
$$
\dim_{\KK_\s}\left( H_{k}(C^\chi_*(\Gamma);\hat\Lambda)\otimes_{\hat\Lambda} {\hat\Lambda}/{(\tau^{s})} \right)-
\dim_{\KK_\s}\left( H_{k}(C^\chi_*(\Gamma);\hat\Lambda)\otimes_{\hat\Lambda} {\hat\Lambda}/{(\tau^{s-1})} \right)=
r_k+\sum_{j\geq s} n_{k,j}(d).
$$
By Proposition~\ref{prop:hpq}
$$
\dim_{\KK_\s}\left( H_{k}(C^\chi_*(\Gamma);{\hat\Lambda}/{(\tau^{s})}) \right)-
\dim_{\KK_\s}\left( H_{k}(C^\chi_*(\Gamma);{\hat\Lambda}/{(\tau^{s-1})}) \right)=
\sum_{i\geq 0} h^s_{\s,(i,k-i)}.
$$
By Proposition~\ref{prop:ns}
$$
\dim_{\KK_\s}\Tor^{\hat\Lambda}_1\left( H_{k-1}(C^\chi_*(\Gamma);\hat\Lambda),\frac{\hat\Lambda}{(\tau^{s})} \right)-
\dim_{\KK_\s}\Tor^{\hat\Lambda}_1\left( H_{k-1}(C^\chi_*(\Gamma);\hat\Lambda),\frac{\hat\Lambda}{(\tau^{s-1})} \right)
=\sum_{j\geq s} n_{k-1,j}(d).
$$
By~\eqref{eq:UCTh} one has
$$
r_k+\sum_{j\geq s} n_{k,j}(d)=\sum_{i\geq 0} h^s_{\s,(i,k-i)}-\sum_{j\geq s} n_{k-1,j}(d).
$$
To obtain~\eqref{eq:thm2} it is enough to use $r_k=\sum_{i\geq 0} h^\infty_{\s,(i,k-i)}$ and induction.
The \emph{moreover} part is a consequence of Proposition~\ref{prop-ss}\ref{prop-ss-3}.
\end{proof}

\begin{exam}{\rm
\label{exam:square-H1}
Consider Figure~\ref{fig:ejemplocuadrado} as a graph whose edges are labeled by the even numbers outside
of the brackets on the edges. The even Artin group $A_\Gamma$ associated with this labeled graph has a presentation
$$
A_\Gamma=\langle g_1,g_2,g_3,g_4: [g_1,g_4]=1, (g_1g_2)^2=(g_2g_1)^2, (g_2g_3)^2=(g_3g_2)^2, (g_3g_4)^2=(g_4g_3)^2\rangle
$$
Consider the character $\chi$ is given in Figure~\ref{fig:ejemplocuadrado} by sending each generator at $v_i$ to the 
corresponding number in parenthesis. The first homology of the Artin kernel $A_\Gamma^\chi$ associated with $\chi$ 
can be studied using the spectral sequence associated with the multiplicity filtration. 
For instance, note that~$\TT_\Gamma=\{2,3,6\}$.

\begin{figure}[ht]
\begin{tikzpicture}
\tikzstyle{pto} = [circle, minimum width=8pt, fill, inner sep=0pt]
\node[pto] (n1) at (3,3) {};
\node[pto] (n2) at (0,0) {};
\node[pto] (n3) at (3,0) {};
\node[pto] (n4) at (0,3) {};
\draw (1.5,0) node[below] {$4,[1,2]$} ;
\draw (1.5,3) node[above] {$4,[1,2]$} ;
\draw (3,1.5) node[right] {$4,[1,2]$} ;
\draw (0,1.5) node[left] {$[0,1],2$} ;
\draw (n1) node[above right] {$(1),[0,0]$} ;
\draw (n1) node[below left] {$v_3$} ;
\draw (n2) node[below left] {$[0,0],(1)$} ;
\draw (n2) node[above right] {$v_1$} ;
\draw (n3) node[below right] {$(2),[0,1]$} ;
\draw (n3) node[above left] {$v_2$} ;
\draw (n4) node[above left] {$[0,1],(2)$} ;
\draw (n4) node[below right] {$v_4$} ;
\draw (n2) -- (n3) -- (n1) -- (n4) -- (n2);
\end{tikzpicture}
\caption{Even Artin graph with character}
\label{fig:ejemplocuadrado}
\end{figure}
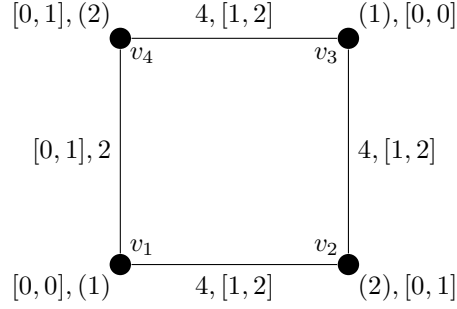
The filtration given by $F_{6,*}$ (resp. $F_{2,*}$) can be summarized by labeling vertices 
and edges of $\Gamma$ with their corresponding weight. This is done in Figure~\ref{fig:ejemplocuadrado}
using the first (resp. second) number in brackets. Therefore, for $F_{6,*}$ note that $E^1_{6,(0,0)}\cong \KK^2$ is generated 
by the cycles $\langle v_2-v_1,v_3-v_1\rangle$ whereas $v_4-v_1$ is the image of the edge $e_{1,4}=\{v_1,v_4\}\in F_{6,0}C_1$.
This spectral sequence degenerates at $E^\infty_{6,(p,q)}=E^2_{6,(p,q)}$. Moreover,
$$E^2_{6,(p,q)}=\begin{cases} \KK & \textrm{ if } p=1, q=0\\ 0 & \textrm{ otherwise.}\end{cases}$$
The only other non-zero term is $E^1_{6,(1,0)}\cong \KK^3$ generated by the edges $\langle e_{1,2},e_{2,3},e_{3,4}\rangle$.
Hence the only non-trivial dimension in the spectral sequence $E^k_{6,(p,q)}$, $k\geq 1$ is $h^1_{6,(0,0)}=2$. 
By Theorem~\ref{thm:main} one obtains
$$
\begin{aligned}
n_{0,1}(6)+n_{0,2}(6)&=\chi^{\text{rel}}_0(E^1_6)=h^1_{6,0}-h^\infty_{6,0}=2,\\
n_{0,2}(6)&=\chi^{\text{rel}}_0(E^2_6)=h^2_{6,0}-h^\infty_{6,0}=0,\\
n_{1,1}(6)+n_{1,2}(6)+n_{1,3}(6)&=\chi^{\text{rel}}_1(E^1_6)=(h^1_{6,1}-h^\infty_{6,1})-(h^1_{6,0}-h^\infty_{6,0})=(3-1)-2=0.\\
\end{aligned}
$$
Thus $n_{0,2}(6)=0$, $n_{0,1}(6)=2$, and $n_{1,1}(6)=n_{1,2}(6)=n_{1,3}(6)=0$.

Analogously, for $F_{2,*}$ note that $E^1_{2,(0,0)}\cong \KK^2$ 
$$
\array{lclclclcl}
0 & \to & E^0_{2,(2,-1)}=\langle e_{1,2},e_{2,3},e_{3,4}\rangle & \to & E^0_{2,(2,-2)}=\{0\} & \to & 0& \to & 0\\
0 & \to & E^0_{2,(1,0)}=\langle e_{1,4}\rangle & \to & E^0_{2,(1,-1)}=\langle v_2,v_4\rangle & \to & 0& \to & 0\\
  &     & e_{1,4}                          & \mapsto & v_4                           &  & \\
0 & \to & E^0_{2,(0,1)}=\{0\} & \to & E^0_{2,(0,0)}=\langle v_{1},v_3\rangle & \to & E^0_{2,(0,-1)}=
\langle \sigma_\emptyset\rangle& \to & 0\\
\endarray
$$
which gives 
$$E^1_{2,(p,q)}=
\begin{cases}
\langle v_3-v_1\rangle & \textrm{ if } p=q=0\\
\langle v_2\rangle & \textrm{ if } p=1,q=-1\\
\langle e_{1,2},e_{2,3},e_{3,4}\rangle & \textrm{ if } p=2,q=-1\\
\{0\} & \textrm{ otherwise.}
\end{cases}
$$
$$E^2_{2,(p,q)}=
\begin{cases}
\langle v_3-v_1\rangle & \textrm{ if } p=q=0\\
\langle e_{1,2}+e_{2,3},e_{3,4}\rangle & \textrm{ if } p=2,q=-1\\
\{0\} & \textrm{ otherwise.}
\end{cases}
$$
$$E^\infty_{2,(p,q)}=E^3_{2,(p,q)}=
\begin{cases}
\langle e_{1,2}+e_{2,3}+e_{3,4}\rangle & \textrm{ if } p=2,q=-1\\
\{0\} & \textrm{ otherwise.}
\end{cases}
$$
In particular, the only non-trivial dimensions in the spectral sequence $E^k_{2,(p,q)}$, $k\geq 1$ are 
$h^1_{2,(0,0)}=h^1_{2,(1,-1)}=h^2_{2,(0,0)}=1$, $h^1_{2,(2,-1)}=3$, $h^2_{2,(2,-1)}=2$, and $h^3_{2,(2,-1)}=1$. 
Hence according to Theorem~\ref{thm:main}
$$
\begin{aligned}
n_{0,1}(2)+n_{0,2}(2)&=\chi^{\text{rel}}_0(E^1_2)=h^1_{2,0}-h^\infty_{2,0}=2,\\
n_{0,2}(2)&=\chi^{\text{rel}}_0(E^2_2)=h^2_{2,0}-h^\infty_{2,0}=1\\
n_{1,1}(2)+n_{1,2}(2)+n_{1,3}(2)&=\chi^{\text{rel}}_1(E^1_2)=(h^1_{2,1}-h^\infty_{2,1})-(h^1_{2,0}-h^\infty_{2,0})=(3-1)-2=0,\\
\end{aligned}
$$
that is, $n_{0,1}(2)=n_{0,2}(2)=1$ and $n_{1,1}(2)=n_{1,2}(2)=n_{1,3}(2)=0$.
Moreover, $\im\partial_{1}=3$, $\im\partial_{2}=0$, and $r_0=0$, $r_1=1$. By Theorem~\ref{teo:torsionhk},
$$
H_1(A^\chi_\Gamma;\KK)=
\left(\frac{\KK[t^{\pm 1}]}{(t-1)}\right)^3\oplus 
\frac{\KK[t^{\pm 1}]}{(t+1)}\oplus 
\frac{\KK[t^{\pm 1}]}{(t+1)^2}\oplus
\left(\frac{\KK[t^{\pm 1}]}{(t^2-t+1)}\right)^2
$$
and 
$$
H_2(A^\chi_\Gamma;\KK)=\KK[t^{\pm 1}].$$}
\end{exam}

\bibliographystyle{amsplain}

\begin{thebibliography}{10}

\bibitem{Antolin}
Y.~Antol\'{i}n and L.~Ciobanu, \emph{Geodesic growth in right-angled and even
  {C}oxeter groups}, European J. Combin. \textbf{34} (2013), no.~5, 859--874.
  \MR{3021516}

\bibitem{ACM1}
E.~Artal, J.I. Cogolludo-Agust\'{\i}n, and D.~Matei, \emph{Quasi-projectivity,
  {A}rtin-{T}its groups, and pencil maps}, Topology of algebraic varieties and
  singularities, Contemp. Math., vol. 538, Amer. Math. Soc., Providence, RI,
  2011, pp.~113--136. \MR{2777818}

\bibitem{ACLMM-Artin}
E.~Artal Bartolo, J.I. Cogolludo-Agust{\'\i}n, S.~L\'opez de~Medrano, and
  D.Matei, \emph{Module structure of the homology of right-angled {A}rtin
  kernels}, 2019.

\bibitem{Blasco-tesis}
R.~Blasco-Garc\'{i}a, \emph{Even {A}rtin groups}, Ph.D. thesis, Universidad de
  Zaragoza, 2019.

\bibitem{Blasco}
R.~Blasco-Garc\'{i}a and J.I. Cogolludo-Agust\'{i}n, \emph{Quasi-projectivity
  of even {A}rtin groups}, Geom. Topol. \textbf{22} (2018), no.~7, 3979--4011.
  \MR{3890769}

\bibitem{Blasco-PF}
R.~Blasco-Garc\'{i}a, C.~Mart\'{i}nez-P\'{e}rez, and L.~Paris,
  \emph{Poly-freeness of even {A}rtin groups of {FC} type}, Groups Geom. Dyn.
  \textbf{13} (2019), no.~1, 309--325. \MR{3900773}

\bibitem{Bourbaki}
N.~Bourbaki, \emph{\'{E}l\'ements de math\'ematique. {F}asc. {XXXIV}. {G}roupes
  et alg\`ebres de {L}ie. {C}hapitre {IV}: {G}roupes de {C}oxeter et syst\`emes
  de {T}its. {C}hapitre {V}: {G}roupes engendr\'es par des r\'eflexions.
  {C}hapitre {VI}: syst\`emes de racines}, Actualit\'es Scientifiques et
  Industrielles, No. 1337, Hermann, Paris, 1968. \MR{0240238}

\bibitem{Bux-Gonzalez-Bestvina}
K-U. Bux and C.~Gonzalez, \emph{The {B}estvina-{B}rady construction revisited:
  geometric computation of {$\Sigma$}-invariants for right-angled {A}rtin
  groups}, J. London Math. Soc. (2) \textbf{60} (1999), no.~3, 793--801.
  \MR{1753814}

\bibitem{Charney-finite}
R.~Charney and M.W. Davis, \emph{Finite {$K(\pi, 1)$}s for {A}rtin groups},
  Prospects in topology ({P}rinceton, {NJ}, 1994), Ann. of Math. Stud., vol.
  138, Princeton Univ. Press, Princeton, NJ, 1995, pp.~110--124. \MR{1368655}

\bibitem{Charney-kpi1}
\bysame, \emph{The {$K(\pi,1)$}-problem for hyperplane complements associated
  to infinite reflection groups}, J. Amer. Math. Soc. \textbf{8} (1995), no.~3,
  597--627. \MR{1303028 (95i:52011)}

\bibitem{Coxeter-discrete}
H.S.M. Coxeter, \emph{Discrete groups generated by reflections}, Ann. of Math.
  (2) \textbf{35} (1934), no.~3, 588--621. \MR{1503182}

\bibitem{Coxeter-complete}
\bysame, \emph{The complete enumeration of finite groups of the form
  ${R}_i^2=({R}_i {R}_j)^{k(i,j)}= 1$}, J. London Math. Soc. \textbf{10}
  (1935), no.~1, 21--25.

\bibitem{Duchamp-PF}
G.~Duchamp and D.~Krob, \emph{Free partially commutative structures}, Journal
  of Algebra \textbf{156} (1993), no.~2, 318--361.

\bibitem{Hermiller-Sunic-PF}
S.~Hermiller and Z.~{\v{S}}uni{\'c}, \emph{Poly-free constructions for
  right-angled {A}rtin groups}, J. Group Theory \textbf{10} (2007), no.~1,
  117--138. \MR{2288463 (2008e:20046)}

\bibitem{Howie}
J.~Howie, \emph{Bestvina-{B}rady groups and the plus construction}, Math. Proc.
  Cambridge Philos. Soc. \textbf{127} (1999), no.~3, 487--493. \MR{1713123
  (2000h:57008)}

\bibitem{learymuge}
I.J. Leary and M.~Saadetoglu, \emph{The cohomology of {B}estvina-{B}rady
  groups}, Preprint available at \texttt{arXiv:0711.5018 [math.AT]} (2007).

\bibitem{MMVW}
J.~Meier, H.~Meinert, and L.~VanWyk, \emph{Higher generation subgroup sets and
  the {$\Sigma$}-invariants of graph groups}, Comment. Math. Helv. \textbf{73}
  (1998), no.~1, 22--44.

\bibitem{Papadima-Suciu-Toric}
S.~Papadima and A.I. Suciu, \emph{Toric complexes and {A}rtin kernels}, Adv.
  Math. \textbf{220} (2009), no.~2, 441--477. \MR{2466422}

\bibitem{Paris}
L.~Paris, \emph{Lectures on {A}rtin groups and the {$K(\pi,1)$} conjecture},
  Groups of exceptional type, {C}oxeter groups and related geometries, Springer
  Proc. Math. Stat., vol.~82, Springer, New Delhi, 2014, pp.~239--257.
  \MR{3207280}

\bibitem{SV:13}
M.~Salvetti and A.~Villa, \emph{Combinatorial methods for the twisted
  cohomology of {A}rtin groups}, Math. Res. Lett. \textbf{20} (2013), no.~6,
  1157--1175.

\end{thebibliography}
\providecommand{\bysame}{\leavevmode\hbox to3em{\hrulefill}\thinspace}
\providecommand{\MR}{\relax\ifhmode\unskip\space\fi MR }
\providecommand{\MRhref}[2]{%
  \href{http://www.ams.org/mathscinet-getitem?mr=#1}{#2}
}
\providecommand{\href}[2]{#2}

\end{document}